\documentclass[12pt]{amsart}

\usepackage{amssymb}

\newtheorem{theorem}{Theorem}[section]
\newtheorem{proposition}[theorem]{Proposition}
\newtheorem{lemma}[theorem]{Lemma}

\theoremstyle{definition}
\newtheorem{definition}[theorem]{Definition}

\newtheorem{example}[theorem]{Example}

\newtheorem{problem}[theorem]{Problem}
\newtheorem{conjecture}[theorem]{Conjecture}

\newcommand{\ZZ}{ \ensuremath{\mathbb{Z}}}

\def\cocoa{{\hbox{\rm C\kern-.13em o\kern-.07em C\kern-.13em o\kern-.15em A}}}

\newcommand{\lk}{\mathrm{lk}}
\newcommand{\ob}{\overline}
\newcommand{\LL}{\mathcal{L}}
\newcommand{\MM}{\mathcal{M}}

\begin{document}

\title[Face vectors of Buchsbaum complexes]{Face vectors of two-dimensional Buchsbaum complexes}

\author{Satoshi Murai}
\address{
Department of Mathematics\\
Graduate School of Science\\
Kyoto University\\
Sakyo-ku, Kyoto, 606-8502, Japan.
}


\begin{abstract}
In this paper, we characterize all possible $h$-vectors of $2$-dimensional Buchsbaum simplicial complexes.
\end{abstract}

\maketitle

\section{Introduction}
Given a class $\mathcal{C}$ of simplicial complexes,
to characterize the face vectors of simplicial complexes in $\mathcal{C}$
is one of central problems in combinatorics.
In this paper, we study face vectors of $2$-dimensional Buchsbaum simplicial complexes.

We recall the basics of simplicial complexes.
A \textit{simplicial complex} $\Delta$ on $[n]=\{1,2,\dots,n\}$
is a collection of subsets of $[n]$
satisfying that (i) $\{i\} \in \Delta$ for all $i \in [n]$
and (ii) if $F \in \Delta$ and $G \subset F$ then $G \in \Delta$.
An element $F$ of $\Delta$ is called a \textit{face} of $\Delta$
and maximal faces of $\Delta$ under inclusion are called \textit{facets} of $\Delta$.
A simplicial complex is said to be \textit{pure}
if all its facets have the same cardinality.
Let $f_k (\Delta)$ be the number of faces $F \in \Delta$ with $|F|=k+1$,
where $|F|$ is the cardinality of $F$.
The \textit{dimension} of $\Delta$ is 
$\dim \Delta = \max \{ k : f_k(\Delta) \ne 0\}$.
The vector $f(\Delta) = (f_{-1}(\Delta),f_0(\Delta), \dots, f_{d-1}(\Delta))$ is called the \textit{$f$-vector} (or \textit{face vector}) of $\Delta$,
where $d= \dim \Delta +1$ and where $f_{-1}(\Delta) =1$.
When we study face vectors of simplicial complexes,
it is sometimes convenient to consider $h$-vectors.
Recall that the {\em $h$-vector} $h(\Delta)=(h_0(\Delta),h_1(\Delta),\dots,h_d(\Delta))$ of
$\Delta$ is  defined by the relation
$\sum_{i=0}^{d} f_{i-1}(\Delta)(x-1)^{d-i}
=\sum_{i=0}^{d} h_{i}(\Delta)x^{d-i}.
$
Thus knowing $f(\Delta)$ is equivalent to knowing $h(\Delta)$.
Let $\tilde H_i(\Delta;K)$ be the reduced homology groups of $\Delta$
over a field $K$. The numbers $\beta_i(\Delta)=\dim_K \tilde H_i(\Delta;K)$ are called the \textit{Betti numbers} of $\Delta$ (over $K$).
The \textit{link} of 
$\Delta$ with respect to $F \in \Delta$ is the simplicial complex
$\lk_\Delta(F)=\{G \subset [n]\setminus F: G \cup F \in \Delta\}$.

In the study of face vectors of simplicial complexes,
one of important classes of simplicial complexes
are Cohen--Macaulay complexes, which come from commutative algebra theory.
A $(d-1)$-dimensional simplicial complex $\Delta$ is said to be \textit{Cohen--Macaulay} if for every face $F \in \Delta$ (including the empty face), $\beta_i(\lk_\Delta(F))=0$ for $i \ne d-1-|F|$.
Given positive integers $a$ and $d$,
there exists the unique representation of $a$,
called the {\em $d$-th Macaulay representation}
of $a$, of the form
\begin{eqnarray*}
a = {a(d) + d  \choose d}
+ {a(d-1) + d - 1 \choose d - 1}
+ \cdots +
{a(k) + k \choose k},
\end{eqnarray*}
where $ k \geq 1$ and where
$a(d) \geq \cdots \geq a(k) \geq 0$.
Define
\begin{eqnarray*}
a^{\langle d \rangle} = {a(d) + d +1 \choose d+1}
+ {a(d-1) + d \choose d}
+ \cdots +
{a(k) + k +1 \choose k+1}
\end{eqnarray*}
and $0^{\langle d \rangle}=0$.
The following classical result due to Stanley \cite[Theorem 6]{St}
has played an important role in face vector theory.

\begin{theorem}(Stanley)
A vector $(1,h_1,\dots,h_d) \in \ZZ^{d+1}$ is the $h$-vector of a $(d-1)$-dimensional Cohen--Macaulay complex if and only if
$h_1 \geq 0$ and $0 \leq h_{i+1} \leq h_i^{\langle i\rangle}$ for $i=1,2,\dots,d-1$.
\end{theorem}

There is another interesting class of simplicial complexes arising from
commutative algebra, called Buchsbaum complexes.
A simplicial complex $\Delta$ is said to be \textit{Buchsbaum}
if it is pure and $\lk_\Delta(v)$ is Cohen--Macaulay for every vertex $v$ of $\Delta$.
Thus the class of Buchsbaum complexes contains
the class of Cohen--Macaulay complexes.
Buchsbaum complexes are important since all triangulations of topological manifolds are Buchsbaum, while most of them are not Cohen--Macaulay.
Several nice necessity conditions on $h$-vectors of Buchsbaum complexes are known (e.g., \cite{Sc,NS}),
and these necessity conditions have been applied to study face vectors of triangulations of manifolds (e.g., \cite{N,NS,Sw}).
On the other hand, the characterization of $h$-vectors of $(d-1)$-dimensional Buchsbaum complexes is a mysterious open problem.
About this problem,
the first non-trivial case is $d=3$ since every $1$-dimensional simplicial complexes (without isolated vertices) are Buchsbaum.
In 1995, Terai \cite{Te} proposed a conjecture on the characterization of $h$-vectors of Buchsbaum complexes of a special type including all $2$-dimensional connected Buchsbaum complexes,
and proved the necessity of the conjecture.
The main result of this paper is to prove the sufficiency of Terai's conjecture for $2$-dimensional Buchsbaum complexes.
As a consequence of this result,
we obtain the following characterizations of $h$-vectors.

\begin{theorem}
\label{1-1}
A vector $h=(1,h_1,h_2,h_3) \in \ZZ^4$ is the $h$-vector of
a $2$-dimensional connected Buchsbaum complex if and only if
the following conditions hold:
\begin{itemize}
\item[(i)] $0 \leq h_1$;
\item[(ii)] $0 \leq h_2 \leq {h_1+1 \choose 2}$;
\item[(iii)] $- \frac 1 3 h_2 \leq h_3 \leq h_2^{\langle 2 \rangle}$.
\end{itemize}
\end{theorem}

\begin{theorem}
\label{1-2}
A vector $h=(1,h_1,h_2,h_3) \in \ZZ^4$ is the $h$-vector of
a $2$-dimensional Buchsbaum complex if and only if
there exist a vector $h'=(1,h_1',h_2',h_3') \in \ZZ^4$ satisfying the conditions in Theorem \ref{1-1} and an integer $k \geq 0$ such that
$h=h'+(0,3k,-3k,k)$.
\end{theorem}

Note that one can always take $k= \lfloor \frac 1 6 (2 h_1+3 - \sqrt{8h_1 +8h_2 +9}) \rfloor = \max\{a : h_2 +3a \leq { {(h_1 -3a) +1} \choose 2}\}$,
where $\lfloor r \rfloor$ is the integer part of a real number $r$.
Indeed, if $(1,h_1',h_2',h_3')$ satisfies the conditions in Theorem \ref{1-1} and if $h_2'+3 \leq { h_1' -3 +1 \choose 2}$ then $(1,h_1'-3,h_2'+3,h_3'-1)$ again satisfies the conditions in Theorem \ref{1-1}.

This paper is organized as follows:
In section 2, some techniques for constructions
of Buchsbaum complexes will be introduced.
In section 3, we construct a Buchsbaum complex with the desired $h$-vector.
In section 4, we prove Theorem \ref{1-2} and
study $h$-vectors of $2$-dimensional Buchsbaum complexes with fixed Betti numbers.

\section{Terai's Conjecture}

We recall Terai's Conjecture \cite[Conjecture 2.3]{Te} on $h$-vectors of Buchsbaum complexes of a special type.
We say that a vector $(1,h_1,\dots,h_d) \in \ZZ^{d+1}$ is an
\textit{$M$-vector} if $h_1 \geq 0$ and $0 \leq h_{i+1} \leq h_i^{\langle i \rangle}$
for $i=1,2,\dots,d-1$.

\begin{conjecture}[Terai]
\label{2-1}
A vector $h=(1,h_1,\dots,h_d) \in \ZZ^{d+1}$ is the $h$-vector of
a $(d-1)$-dimensional Buchsbaum complex $\Delta$ such that $\beta_k(\Delta)=0$ for $k \leq d-3$ if and only if
the following conditions hold:
\begin{itemize}
\item[(a)] $(1,h_1,\dots,h_{d-1})$ is an $M$-vector;
\item[(b)] $- \frac 1 d h_{d-1} \leq h_d \leq h_{d-1}^{\langle d-1 \rangle}$.
\end{itemize}
\end{conjecture}

Terai \cite{Te} proved the `only if' part of the above conjecture.
Thus the problem is to construct a Buchsbaum complex $\Delta$ such that
$\beta_k(\Delta)=0$ for $k \leq d-3$ and $h(\Delta)=h$.
Actually, if $h_d \geq 0$ then any vector $h \in \ZZ^4$ satisfying
(a) and (b) is an $M$-vector, so there exists a Cohen--Macaulay complex $\Delta$ with $h(\Delta)=h$ by Stanley's theorem.
Thus it is enough to consider the case when $h_d<0$.

From this viewpoint, Terai \cite{Te} and Hanano \cite{Ha} constructed a class of $2$-dimensional Buchsbaum complexes $\Delta$ with $h_3(\Delta)= -\frac 1 3 h_{2}(\Delta)$.
Also, by using Hanano's result, Terai and Yoshida \cite{TY}
proved the conjecture in the special case when $d=3$ and $h_2={h_1+1 \choose 2}$.

In this paper, we prove Conjecture \ref{2-1} when $d=3$,
which is equivalent to Theorem \ref{1-1}.
Since we only need to consider the case when $h_3 <0$,
what we must prove is the following statement.

\begin{proposition}
\label{2-2}
Let $h_1$, $h_2$ and $w$ be positive integers such that $3 w \leq h_2 \leq {h_1 +1 \choose 2}$.
There exists a $2$-dimensional connected Buchsbaum complex $\Delta$ such that
$h(\Delta)=(1,h_1,h_2,-w)$.
\end{proposition}

In the rest of this section,
we introduce techniques to prove the above statement.
We first note the exact relations between $f$-vectors and $h$-vectors
when $d=3$.
\begin{eqnarray*}
\begin{array}{llll}
h_0=1, &h_1=f_0-3, &h_2=f_1-2f_0+3, &h_3= f_2-f_1+f_0-1,\\
f_{-1}=1, & f_0=h_1+3, & f_1=h_2+2h_1 +3, & f_2=h_3+h_2+h_1+1.\\
\end{array}
\end{eqnarray*}

\begin{lemma}
\label{2-3}
Let $\Delta$ be a $(d-1)$-dimensional Buchsbaum complex on $[n]$.
Then $h_d(\Delta)=-\frac 1 d h_{d-1}(\Delta)$ if and only if,
for every $v \in [n]$, $\beta_k(\lk_\Delta(v))=0$ for all $k$.
\end{lemma}

\begin{proof}
The statement follows from the next computation.
\begin{eqnarray*}
d h_d +h_{d-1} &=& \sum_{k=0}^{d} (-1)^{d-k} k f_{k-1}(\Delta)\\
&=& \sum_{v \in [n]} \left( \sum_{k=0}^{d-1} (-1)^{d-1-k} f_{k-1}\big(\lk_\Delta(v) \big) \right)\\
&=& \sum_{v \in [n]} \beta_{d-2}\big( \lk_\Delta(v) \big). 
\end{eqnarray*}
The second equation follows from $\sum_{v \in [n]} f_{k-2}(\lk_\Delta(v))
=k f_{k-1}(\Delta)$, 
and, for the third equation,
we use the Buchsbaum property together with the well-known equation
$\sum_{k=0}^{d-1}(-1)^{d-1-k}\beta_{k-1}(\lk_\Delta(v))
=\sum_{k=0}^{d-1}(-1)^{d-1-k}f_{k-1}(\lk_\Delta(v))$.
\end{proof}

\begin{definition}
We say that a Buchsbaum complex $\Delta$ on $[n]$ is \textit{link-acyclic} if $\Delta$ satisfies one of the conditions in Lemma \ref{2-3}.
\end{definition}

Every $1$-dimensional simplicial complex is identified with a simple graph, and, in this special case,
the Cohen--Macaulay property is equivalent to the connectedness.
Thus a $2$-dimensional pure simplicial complex is Buchsbaum if and only if its every vertex link is a connected graph.
Moreover, a $2$-dimensional Buchsbaum complex is link-acyclic if and only if its every vertex link is a tree.
From this simple observation, it is easy to prove the following statements.

\begin{lemma}
\label{2-4}
Let $\Delta$ be a $2$-dimensional Buchsbaum complex  on $[n]$ and let $\Delta_1,\dots,\Delta_t$ be $2$-dimensional simplicial complexes whose vertex set is contained in $[n]$.
\begin{itemize}
\item[(i)] If $\Delta \cup \Delta_k$ is Buchsbaum for $k=1,2,\dots,t$ then $\Delta \cup \Delta_1 \cup \cdots \cup \Delta_j$
is also Buchsbaum for $j=1,2,\dots,t$.
\item[(ii)] If $\Delta$ is link-acyclic then any $2$-dimensional Buchsbaum complex $\Gamma \subset \Delta$ is also link-acyclic.
\end{itemize}
\end{lemma}

\begin{proof}
(i) Without loss of generality we may assume $j=t$.
Let $\Sigma = \Delta \cup \Delta_1 \cup \cdots \cup \Delta_t$.
Fix $v \in [n]$.
What we must prove is $\lk_\Sigma(v)$ is connected.
Let $v_0$ be a vertex of $\lk_\Delta (v)$.
For every vertex $u$ of $\lk_\Sigma(v)$ there exists a $k$
such that $u$ is a vertex of $\lk_{\Delta \cup \Delta_k}(v)$.
By the assumption, there exists a sequence $u=u_0,u_1,\dots,u_r=v_0$ such that $\{u_i,u_{i+1}\} \in \lk_{\Delta \cup \Delta_k}(v) \subset \lk_\Sigma(v)$ for $i=0,1,\dots,r-1$.
Hence $\lk_\Sigma(v)$ is connected.

(ii) For every vertex $v$ of $\Gamma$,
$\lk_\Gamma(v)$ is connected and $\lk_\Gamma(v) \subset \lk_\Delta(v)$.
Since $\lk_\Delta(v)$ is a tree, $\lk_\Gamma(v)$ is also a tree.
\end{proof}

For a collection $C$ of subsets of $[n]$,
we write $\langle C \rangle$ for the simplicial complex generated by the elements in $C$.
When $C=\{F\}$, we simply write $\langle C \rangle =\langle F \rangle$.

\begin{lemma}
\label{2-5}
Let $\Delta$ be a $2$-dimensional Buchsbaum complex and $F=\{a,b,c\}$.
Set $\Gamma = \Delta \cup \langle F\rangle$.
\begin{itemize}
\item[(i)] If $\Delta \cap \langle F\rangle= \langle \{a,b\},\{a,c\},\{b,c\}\rangle$ then
$\Gamma$ is Buchsbaum and $h(\Gamma)=h(\Delta)+(0,0,0,1)$.
\item[(ii)] If $\Delta \cap \langle F\rangle= \langle\{a,b\},\{a,c\}\rangle$ then
$\Gamma$ is Buchsbaum and $h(\Gamma)=h(\Delta)+(0,0,1,0)$.
\item[(iii)] If $\Delta \cap \langle F\rangle= \langle\{a,b\}\rangle$ then
$\Gamma$ is Buchsbaum and $h(\Gamma)=h(\Delta)+(0,1,0,0)$.
\end{itemize}
\end{lemma}

\section{Proof of Proposition \ref{2-2}}

In this section, we prove Proposition \ref{2-2}.
Let $h=(1,h_1,h_2,-w) \in \ZZ^4$ be the vector satisfying $w>0$ and $3w \leq h_2 \leq {h_1 +1 \choose 2}$.

Let $x$ be the smallest integer $k$ such that $3w \leq {k +1 \choose 2}$ and $y=\min\{h_2,{x+1 \choose 2}\}$.
We write
$$h=(1,x,y,-w) +(0,\gamma,\delta,0).$$
Then the vector $(1,x,y,-w)$ again satisfies the conditions in Proposition \ref{2-2} (that is, $3w \leq y \leq {x+1 \choose 2}$).
Also, if $\delta >0$ then $y={x+1 \choose 2}$.
The next lemma shows that, to prove Proposition \ref{2-2},
it is enough to consider the vector $(1,x,y,-w)$.

\begin{lemma}
\label{3-1}
If there exists a $2$-dimensional connected Buchsbaum complex $\Delta$
such that $h(\Delta)=(1,x,y,-w)$ then there exists a $2$-dimensional connected Buchsbaum complex $\Gamma$ such that $h(\Gamma)=h$.
\end{lemma}

\begin{proof}
We may assume that $\Delta$ is a simplicial complex on $[x+3]$ such that
$\{1,2\} \in \Delta$.
For $j=0,1,\dots,\gamma$,
let
$$\Delta_j=\Delta \cup \big\langle 
\big\{ \{1, 2, x+3+k\}: k=1,2,\dots,j \big\} \big\rangle,$$
where $\Delta_0=\Delta$.
Since $\Delta_{j-1} \cap \langle \{ 1, 2, {x+3+j}\}\rangle = \langle \{1,2\} \rangle$,
Lemma \ref{2-5}(iii) says that $\Delta_\gamma$ is a connected Buchsbaum complex with $h(\Delta_\gamma)=(1,x+\gamma,y,-w)$.

If $\delta=0$ then $\Delta_\gamma$ satisfies the desired conditions.
Suppose $\delta>0$.
Then $y={x+1 \choose 2}$. This means that $\Delta$ contains all $1$-dimensional simplexes $\{i,j\} \subset [x+3]$.
Let
$$E= \big\{ \{i,j\} \subset \{3,4,\dots,x+\gamma+3\}:
\{i,j\} \not \subset [x+3],\ i \ne j\big\}.$$
Then $E$ is the set of $1$-dimensional non-faces of $\Delta_\gamma$.
Also,
$$\delta=h_2 -y \leq {x+\gamma +1 \choose 2} -{x+1 \choose 2} =|E|.$$
Choose distinct elements $\{i_1,j_1\},\{i_2,j_2\},\dots,\{i_\delta,j_\delta\} \in E$.
Let
$$\Gamma_\ell=\Delta_\gamma \cup \big\langle  \big\{\{1,i_k,j_k\}:k=1,2,\dots,\ell \big\}\big\rangle $$
for $\ell=0,1,\dots,\delta$, where $\Gamma_0=\Delta_\gamma$.
Since $\Gamma_{\ell-1} \cap \langle \{1, i_\ell, j_\ell\}\rangle=\langle\{1,i_\ell\}, \{1,j_\ell\}\rangle$,
it follows from Lemma \ref{2-5}(ii) that $\Gamma_\delta$ is a connected Buchsbaum complex with $h(\Gamma_\delta)=(1,x+\gamma,y+\delta,-w)=h$.
\end{proof}

Let $n=x+3$ and $M=\max\{k:3k \leq {x+1 \choose 2}\}$.
Write $n=3p+q$ where $p \in \ZZ$ and $q \in \{0,\pm 1\}$.
Then
\begin{eqnarray*}
M= \left\{
\begin{array}{lll}
\frac 1 3 {n-2 \choose 2} = \frac 1 2 (p-1)(3p-4),&
\mbox{ if }n=3p-1,\smallskip\\
\frac 1 3 {n-2 \choose 2} = \frac 1 2 (p-1)(3p-2),&
\mbox{ if }n=3p,\smallskip\\
\frac 1 3 \big\{ {n-2 \choose 2} -1\big\} = \frac 1 2 (p-1)3p,&
\mbox{ if }n=3p+1.
\end{array}
\right.
\end{eqnarray*}
Let $b$, $c$ and $\alpha$ be non-negative integers satisfying
$$(1,x,y,-w)=\big(1,n-3,3(M-b)+\alpha,-(M-b)+c\big)$$
and $\alpha\in \{0,1,2\}$ ($\alpha$ is the remainder of $y/3$).
Since ${x \choose 2} <3w \leq {x+1 \choose 2}$ by the choice of $x$,
the following conditions hold:
\begin{itemize}
\item $n \geq 5$ and $p \geq 2$;
\item $0 \leq b+c \leq p-2$.
\end{itemize}
Note that $n \geq 5$ holds since $3w \leq {x+1 \choose 2}$
and $w$ is positive.
Also, $b+c \leq p-2$ holds since if $b+c \geq p-1$ then
$3 w =3(M-b-c) \leq {x \choose 2}$.

We will construct a Buchsbaum complex $\Delta$ on $[n]$ with $h(\Delta)=(1,x,y,-w)$.
The construction depends on the remainder of $n/3$,
and will be given in subsections 3.1, 3.2 and 3.3.
We explain the procedure of the construction.
First, we construct a connected Buchsbaum complex $\Gamma$ with the $h$-vector
$(1,n-3,3(M-b),-(M-b))$.
Second, we construct a Buchsbaum complex $\Delta$ with the $h$-vector
$(1,n-3,3(M-b),-(M-b)+c)$ by adding certain $2$-dimensional simplexes to $\Gamma$ and by applying Lemma \ref{2-5}(i).
Finally, we construct a Buchsbaum complex with the desired $h$-vector
by using Lemma \ref{2-5}(ii).
\bigskip

\noindent
\textbf{Remarks and Notations of subsections 3.1, 3.2 and 3.3.}
For an integer $i \in \ZZ$ we write $\overline i$ for the integer in $[n]$ such that $\overline i \equiv i$ mod $n$.
The constructions given in subsections 3.1, 3.2 and 3.3 are different,
however, the proofs are similar.
Thus we write details of proofs in subsection 3.1
and sketch proofs in subsections 3.2 and 3.3.

\subsection{Construction when $n=3p-1$}
\ \\
Let
$$\Sigma = \big\langle \big\{ \{\ob i, \ob {1+i}, \ob{2+i}\}:i=1,2,\dots,n \big\} \big\rangle $$
and for $i=1,2,\dots,n$ and $j=1,2,\dots,p-2$, let
$$\Delta(i,j)= \big\langle \{ \ob i, \ob{1+i},\ob {2+i+3j}\}, \{\ob{1+i+3j}, \ob{2+i+3j},\ob{1+i}\}\big\rangle .$$
Let
$$\LL =\big\{ \Delta(i,j): i=1,2,\dots,n \mbox{ and } j=1,2,\dots,p-2\big\}$$
and
$$\hat \Delta = \Sigma \cup \left( \bigcup_{\Delta(i,j) \in \LL} \Delta(i,j) \right).$$
Then it is easy to see that
\begin{itemize}
\item $\Delta(i,j)=\Delta(1+i+3j,p-1-j)$.\smallskip
\item If $\Delta(i,j)\ne \Delta(i',j')$ then
$\Delta(i,j)$ and $\Delta(i',j')$ have no common facets. \smallskip
\item $\hat \Delta= \langle \{ \{\ob i, \ob{1+i},\ob{2+i+3j}\}:i=1,2,\dots,n \mbox{ and }j=0,1,\dots,p-2\} \rangle.$
\end{itemize}

\begin{example}
\label{exa1}
Consider the case when $n=8$.
Then $p=3$ and
$$\Sigma=\big\langle \{1,2,3\},\{2,3,4\},\{3,4,5\},\{4,5,6\},\{5,6,7\},\{6,7,8\},\{7,8,1\},\{8,1,2\}\big\rangle.$$
Also,
\begin{eqnarray*}
\begin{array}{ll}
\Delta(1,1)=\Delta(5,1)=\langle \{1,2,6\},\{5,6,2\} \rangle,\
&\Delta(2,1)=\Delta(6,1)=\langle \{2,3,7\},\{6,7,3\} \rangle,\\
\Delta(3,1)=\Delta(7,1)=\langle \{3,4,8\},\{7,8,4\} \rangle,\
&\Delta(4,1)=\Delta(8,1)=\langle \{4,5,1\},\{8,1,5\} \rangle.
\end{array}
\end{eqnarray*}
\end{example}

\begin{lemma}
\label{3-2}
\ 
\begin{itemize}
\item[(i)] (Hanano) $\hat \Delta$ is Buchsbaum, link-acyclic and $h(\hat \Delta)=(1,n-3,3M,-M)$.
\item[(ii)] For any subset $\MM \subset \LL$,
$\Sigma \cup(\bigcup_{\Delta(i,j)\in\MM}\Delta(i,j))$ is Buchsbaum and link-acyclic.
\end{itemize}
\end{lemma}

\begin{proof}
The simplicial complex $\Sigma$ is Buchsbaum since
its every vertex link is connected.
Also, for any $\Delta(i,j) \in \LL$,
one can easily see that every vertex link of $\Sigma \cup \Delta(i,j)$
is connected.
Then the Buchsbaum property of (i) and (ii) follows from
Lemma \ref{2-4}(i).

To prove the link-acyclic property of (i) and (ii),
what we must prove is that $\hat \Delta$ is link-acyclic by Lemma \ref{2-4}(ii).
It is enough to prove $h(\hat \Delta)=(1,n-3,3M,-M)$,
equivalently $f(\hat \Delta)=(1,n,{n \choose 2},2M+n-2)$.
This fact was shown in \cite{Ha}.
Thus we sketch the proof.
It is clear that $f_2(\hat \Delta)=n(p-1)=2M+n-2$.
On the other hand, $f_1(\hat \Delta)={n \choose 2}$ holds since $\hat \Delta$ contains
all $1$-dimensional faces $\{i,j\} \subset [n]$.
\end{proof}

Recall that what we want to do is to construct a connected Buchsbaum complex with the $h$-vector
$(1,n-3,3(M-b)+\alpha,-(M-b)+c)$, where $\alpha \in \{0,1,2\}$ and $b+c \leq p-2$.
Let
$$\MM = \LL \setminus \{\Delta(1,1),\Delta(1,2),\dots,\Delta(1,b)\}$$
and
$$\Gamma = \Sigma \cup \left ( \bigcup_{\Delta(i,j) \in \MM} \Delta(i,j) \right).$$
For $j=1,2,\dots,p-2$, let
$$G_j=\{1,2+3j,3+3j\}.$$
Note that $G_j \not \in \hat \Delta$.
Define
$$
\Delta_k=\Gamma \cup \langle G_{b+1} \rangle \cup \langle G_{b+2} \rangle
\cup \cdots \cup \langle G_{b+k} \rangle
$$
for $k=0,1,\dots,c$, where $\Delta_0=\Gamma$.

\begin{lemma}
\label{3-3}
For $k=0,1,\dots,c$,
the simplicial complex $\Delta_k$ is connected, Buchsbaum and
$h(\Delta_k)=(1,n-3,3(M-b),-(M-b)+k)$.
\end{lemma}

\begin{proof}
The connectedness is obvious.
By Lemma \ref{3-2}, $\Gamma$ is Buchsbaum and link-acyclic.
In particular, since $f_2(\Gamma)=f_2(\hat \Delta)-2b$,
the equation $f_2=h_0+h_1+h_2+h_3$ and the link-acyclic property imply
$$h(\Gamma)=h(\hat \Delta)-(0,0,3b,-b)=\big(1,n-3,3(M-b),-(M-b)\big).$$
Then, to complete the proof,
by Lemma \ref{2-5}(i) it is enough to prove that
\begin{eqnarray*}
\
\Delta_{k-1} \cap \langle G_{b+k}\rangle = \big\langle \{1,2+3(b+k)\}, \{1,3+3(b+k)\},\{2+3(b+k),3+3(b+k)\} \big\rangle 
\end{eqnarray*}
for $k=1,2,\dots,c$.
It is clear that $G_{b+k} \not \in \Delta_{k-1}$.
Also, $\{1,3+3(b+k)\}, \{2+3(b+k),3+3(b+k)\} \in \Delta(1,b+k) \subset \Delta_{k-1}$.
Finally,
$\{1,2+3(b+k)\} \in \Delta(n,b+k)$ and $\Delta(n,b+k) \subset \Gamma \subset \Delta_{k-1}$
by the construction of $\Gamma$.
\end{proof}

Let $\Delta=\Delta_c$.
If $\alpha=0$ then $\Delta$ has the desired $h$-vector.
We consider the case $\alpha\in\{1,2\}$.
Then $b>0$ since $3(M-b) + \alpha \leq {n-2 \choose 2}=3M$.
The next lemma and Lemma \ref{2-5}(ii) guarantee the existence of a $2$-dimensional
connected Buchsbaum complex with the $h$-vector
$(1,x,y,-w)=(1,n-3,3(M-b)+\alpha,-(M-b)+c)$.

\begin{lemma}
\label{3-4}\ 
\begin{itemize}
\item[(i)] $\Delta \cap \langle G_b\rangle = \langle \{1,2+3b\},\{2+3b,3+3b\}\rangle$.
\item[(ii)] $(\Delta\cup \langle G_b\rangle)\cap \langle \{1,2,3+3b\}\rangle=\langle \{1,2\},\{1,3+3b\}\rangle$.
\end{itemize}
\end{lemma}

\begin{proof}
First, we claim that $\{1,3+3b\},\{2,2+3b\},\{2,3+3b\} \not \in\Delta$.
By Lemma \ref{3-2}, both $\Gamma$ and $\Gamma \cup \Delta(1,b)$ are Buchsbaum and link-acyclic.
Since $\Delta(1,b) \not \subset \Gamma$, $f_2(\Gamma\cup \Delta(1,b))=f_2(\Gamma)+2$.
Then the link-acyclic property shows $h(\Gamma\cup \Delta(1,b))=h(\Gamma)+(0,0,3,-1)$.
This fact implies $f_1(\Gamma\cup \Delta(1,b))=f_1(\Gamma)+3$.
Thus $\Delta(1,b)$ contains three edges which are not in $\Gamma$.
Actually, $\Delta(1,b)$ has $5$ edges
$$\{1,2\},\{2+3b,3+3b\},\{1,3+3b\},\{2,2+3b\},\{2,3+3b\}.$$
Since the first two edges are contained in $\Sigma$,
the latter three edges are not contained in $\Gamma$.
Since $h_i(\Gamma)=h_i(\Delta)$ for $i \leq 2$,
$f_1(\Gamma)=f_1(\Delta)$.
Thus the set of edges in $\Gamma$ and that of $\Delta$ are same.
Hence $\{1,3+3b\},\{2,2+3b\},\{2,3+3b\} \not \in\Delta$ as desired.

Then (i) holds since $\{1,2+3b\} \in \Delta(n,b) \subset \Delta$,
$\{2+3b,3+3b\} \in \Sigma \subset \Delta$ and $\{1,3+3b\} \not \in \Delta$,
and (ii) holds since $\{1,2\} \in \Sigma \subset \Delta$,
$\{1,3+3b\} \in \langle G_b\rangle$ and $\{2,3+3b\} \not \in \Delta \cup \langle G_b\rangle$.
\end{proof}

\begin{example}
Again, consider the case when $n=8$ as in Example \ref{exa1}.
In this case, $M=5$.
We construct a $2$-dimensional Buchsbaum complex with the $h$-vector
$(1,n-3,3(M-1)+2,-(M-1))=(1,5,14,-4)$.

The simplicial complex $\Delta_0=\Gamma=\Sigma \cup \Delta(2,1) \cup \Delta(3,1)\cup \Delta(4,1)$ is Buchsbaum and $h(\Gamma)=(1,5,12,-4)$.
Now,
$G_1=\{1,5,6\}$ and
$\Delta_0 \cup \langle \{1,5,6\} \rangle$ has the $h$-vector
$(1,5,13,-4)$.
Finally, $\Delta_0 \cup \langle \{1,5,6\},\{1,2,6\} \rangle$
has the $h$-vector $(1,5,14,-4)$ as desired.
\end{example}

\subsection{Construction when $n=3p$}\ \\
Let
$$\Sigma = \big\langle \big\{ \{\ob i, \ob{i+p}, \ob{i+2p}\}: i=1,2,\dots,p \big\} \big\rangle $$
and, for $i=1,2,\dots,n$ and $j=1,2,\dots,p-1$, let
$$\Delta(i,j)=\big\langle \{\ob i, \ob{i+p},\ob{i+j+p}\},\{\ob{i+j+p},\ob{i+j+2p},\ob i\}\big\rangle .$$
Let
$$\LL=\big\{ \Delta(i,j):i=1,2,\dots,n \mbox{ and }j=1,2,\dots,p-1\big\}$$
and
$$\hat \Delta = \Sigma \cup \left( \bigcup_{\Delta(i,j) \in \LL} \Delta(i,j) \right).$$
Note that
$$\hat \Delta = \Sigma \cup \big\langle \big\{ \{\ob i, \ob{i+p}, \ob{i+j+p}\}: i=1,2,\dots,n \mbox{ and }j=1,2,\dots,p-1\big\} \big\rangle .$$
The next lemma can be proved in the same way as in Lemma \ref{3-2}.

\begin{lemma}
\label{3-5}
\ \begin{itemize}
\item[(i)] (Hanano) $\hat \Delta$ is Buchsbaum, link-acyclic and $h(\hat \Delta)=(1,n-3,3M,-M)$.
\item[(ii)] For any subset $\MM \subset \LL$,
$\Sigma \cup(\bigcup_{\Delta(i,j)\in\MM}\Delta(i,j))$ is Buchsbaum and link-acyclic.
\end{itemize}
\end{lemma}

Let
$\MM = \LL \setminus \{\Delta(1,1),\Delta(1,2),\dots,\Delta(1,b)\}$
and
$$\Gamma = \Sigma \cup \left ( \bigcup_{\Delta(i,j) \in \MM} \Delta(i,j) \right).$$
For $j=1,2,\dots,p-2$, let
$$G_j=\{1+p,1+j+p,1+j+2p\}.$$
Note that $G_j \not \in \hat \Delta$.
Define
$$
\Delta_k=\Gamma \cup \langle G_{b+1} \rangle \cup \langle G_{b+2} \rangle
\cup \cdots \cup \langle G_{b+k} \rangle
$$
for $k=0,1,\dots,c$, where $\Delta_0=\Gamma$.

\begin{lemma}
\label{3-6}
For $k=0,1,\dots,c$,
the simplicial complex $\Delta_k$ is connected, Buchsbaum and
$h(\Delta_k)=(1,n-3,3(M-b),-(M-b)+k)$.
\end{lemma}

The proof of the above lemma is the same as that of Lemma \ref{3-3}.
(To prove that $\Delta_{k-1}\cap\langle G_{b+k}\rangle$ is generated by three edges,
use $\{1+p,1+(b+k)+p\},\{1+(b+k)+p,1+(b+k)+2p\} \in \Delta(1,b+k) \subset \Gamma$ and $\{1+p,1+(b+k)+2p\} \in \Delta(1+p,b+k) \subset \Gamma$.)

Let $\Delta =\Delta_c$.
Then the next lemma and Lemma \ref{2-5}(ii) guarantee the existence of a $2$-dimensional connected Buchsbaum complex
with the $h$-vector
$(1,x,y,-w)=(1,n-3,3(M-b)+\alpha,-(M-b)+c)$.

\begin{lemma}
\label{3-7}\ 
\begin{itemize}
\item[(i)] $\Delta \cap \langle G_b\rangle = \langle \{1+p,1+b+2p\},\{1+b+p,1+b+2p\}\rangle$.
\item[(ii)] $(\Delta\cup \langle G_b\rangle)\cap \langle \{1,1+p,1+b+p\}\rangle=\langle \{1,1+p\},\{1+p,1+b+p\}\rangle$.
\end{itemize}
\end{lemma}

\begin{proof}
By using Lemmas \ref{3-5} and \ref{3-6},
one can prove $f_1(\Gamma \cup \Delta(1,b))=f_1(\Gamma)+3$
in the same way as in the proof of Lemma \ref{3-4}.
The complex $\Delta(1,b)$ has $5$ edges
$$\{1,1+p\},\{1+b+p,1+b+2p\},\{1,1+b+p\},\{1,1+b+2p\},\{1+p,1+b+p\}.$$
Since the first two edges are contained in $\Sigma \subset \Gamma$,
the latter three edges are not contained in $\Gamma$.
Since $h_i(\Gamma)=h_i(\Delta)$ for $i \leq 2$,
the set of edges in $\Gamma$ and that of $\Delta$ are same.
Hence these three edges are not in $\Delta$.

Then (i) holds since $\{1+p,1+b+2p\} \in \Delta(1+p,b) \subset \Delta$,
$\{1+b+p,1+b+2p\} \in \Sigma \subset \Delta$ and $\{1+p,1+b+p\} \not \in \Delta$,
and (ii) holds since $\{1,1+p\} \in \Sigma \subset \Delta$,
$\{1+p,1+b+p\} \in \langle G_b\rangle$ and $\{1,1+b+p\} \not \in \Delta \cup \langle G_b\rangle$.
\end{proof}

\subsection{Construction when $n=3p+1$}\ \\
Let
$$\Sigma = \big\langle \big\{ \{\ob{i-(p-1)}, \ob{i},\ob{i+(p+1)}\}:
i=1,2,\dots,n\big\} \big\rangle.$$
For $i=1,2,\dots,n$ and $j=1,2,\dots,p-2$,
let
$$\Delta(i,j)= \big\langle \{ \ob{i-j},\ob{i},\ob{i+(2p-j)}\},
\{\ob{1+i+p},\ob{i+(2p-j)},\ob{i}\} \big\rangle$$
and for $i=1,2,\dots,p$ let
$$\Delta(i,\infty)=\big\langle \{\ob{i},\ob{i+p},\ob{i+2p}\},
\{\ob{i+p},\ob{i+2p},\ob{i+3p}\}\big\rangle.$$
Let
$$\LL=\big\{
\Delta(i,j): i=1,2,\dots,n \mbox{ and }j=1,2,\dots,p-2 \big\}
\cup \big\{\Delta(i,\infty):i=1,2,\dots,p\big\}$$
and
$$
\hat \Delta = \Sigma \cup\left( \bigcup_{\Delta(i,j) \in \LL} \Delta(i,j) \right).$$
Note that
\begin{eqnarray}
\label{4.b}
&&\Sigma \cup \left(\bigcup_{ 1\leq i \leq n,\ 1 \leq j\leq p-2}\Delta(i,j)\right)\\
\nonumber
&&=\big\langle \big\{\{\ob{i-j},\ob i, \ob{i+(2p-j)}\}:i=1,2,\dots,n 
\mbox{ and }j=1,2,\dots,p-1\big\} \big\rangle.
\end{eqnarray}

\begin{lemma}
\label{3-8}
\ 
\begin{itemize}
\item[(i)] $\hat \Delta$ is Buchsbaum, link-acyclic, $h(\hat \Delta)=(1,n-3,3M,-M)$ and $\{p,n \} \not \in \hat \Delta$.
\item[(ii)] For any subset $\MM \subset \LL$, 
$\Sigma \cup(\bigcup_{\Delta(i,j)\in\MM}\Delta(i,j))$ is Buchsbaum and link-acyclic.
\end{itemize}
\end{lemma}

\begin{proof}
The simplicial complex $\Sigma$ is Buchsbaum since its every vertex link
is connected.
Also, for any $\Delta(i,j) \in \LL$,
a routine computation shows that every vertex link of $\Sigma\cup\Delta(i,j)$ is connected.
Then the Buchsbaum property of (i) and (ii) follows from Lemma \ref{2-4}(i).

To prove the link-acyclic property,
it is enough to prove that $\hat \Delta$ is link-acyclic by Lemma \ref{2-4}(ii).
We will show $h(\hat \Delta)=(1,n-3,3M,-M)$,
equivalently $f(\hat \Delta)=(1,n,{n \choose 2}-1,2M+n-2)$.
It is easy to see that $f_2(\hat \Delta)=n(p-1)+2p=2M+n-2$.
We will show $f_1(\hat \Delta)={n \choose 2}-1$.

We claim that $\hat \Delta$ contains every $\{i,j\} \subset [n]$ except for $\{p,n \}$.
For any $\{i,j \} \subset [n]$,
there exists a $1 \leq k \leq 2p-1$ such that $j= \ob{i+k}$ or $i=\ob{j+k}$. We may assume $j= \ob{i+k}$.
If $k \ne p$ then we have either $\{i,j\} \in \Sigma$ or $\{i,j\} \in \Delta(i',j')$ for some $i',j'$ with $j' \ne \infty$ by (\ref{4.b}).
On the other hand, for every $1 \leq i \leq n-1$,
we have $\{\ob i, \ob{i+p}\} \in \Delta(i',\infty)$ for some $i'$.
Thus $\hat \Delta$ contains every $\{i,j\} \subset [n]$
such that $\{i,j\} \ne \{p,n\}$.
Finally,
since $\Sigma$ and any $\Delta(i,j) \in \LL$ with $j \ne \infty$
contain no elements of the form $\{\ob i, \ob{i+p}\}$
and since $\{p,n\}=\{n,\ob{n+p}\} \not \in \Delta(i',\infty)$ for $i'=1,2,\dots,p$,
we have $\{p,n\} \not \in \hat \Delta$ as desired.
\end{proof}

Let 
$\MM = \LL \setminus \{\Delta(1,1),\Delta(1,2),\dots,\Delta(1,b)\}$
and
$$\Gamma = \Sigma \cup \left ( \bigcup_{\Delta(i,j) \in \MM} \Delta(i,j) \right).$$
For $j=1,2,\dots,p-2$, let
$$G_j=\{\ob{1-j},1,2+p\}.$$
Note that $G_j \not \in \hat \Delta$.
Define
$$
\Delta_k=\Gamma \cup \langle G_{b+1} \rangle \cup \langle G_{b+2} \rangle
\cup \cdots \cup \langle G_{b+k} \rangle
$$
for $k=0,1,\dots,c$, where $\Delta_0=\Gamma$.

\begin{lemma}
\label{3-9}
For $k=0,1,\dots,c$,
the simplicial complex $\Delta_k$ is connected, Buchsbaum
and $h(\Delta_k)=(1,n-3,3(M-b),-(M-b)+k)$.
\end{lemma}

The proof of the above lemma is the same as that of Lemma \ref{3-3}.
(To prove $\Delta_{k-1} \cap \langle G_{b+k} \rangle$ is generated by three edges, use
$\{\ob{1-(b+k)},1\},\{1,2+p\} \in \Delta(1,b+k) \subset \Gamma$
and $\{\ob{1-(b+k)},2+p\} \in \Delta(2+p,b+k) \subset \Gamma$.)

Let $\Delta=\Delta_c$.
We will construct a $2$-dimensional
connected Buchsbaum complex with the $h$-vector
$(1,x,y,-w)=(1,n-3,3(M-b)+\alpha,-(M-b)+c)$.
If $\alpha=0$ then $\Delta$ satisfies the desired conditions.
Suppose $\alpha>0$.\medskip

\textit{Case 1}:
If $b=0$ then $\alpha=1$ and $\Gamma=\hat \Delta$ since $3M={n-2 \choose 2}-1$ and $y \leq {n-2 \choose 2}$.
Since the set of edges in $\Delta$ and that of $\Gamma=\hat \Delta$ 
are same and since $\{p,n\} \not \in \hat \Delta$,
it follows that $\Delta \cap \langle \{1,p,n\}\rangle =
\langle \{1,p\}, \{1,n\}\rangle$.
By Lemma \ref{2-5}(ii),
$\Delta \cup \langle \{1,p,n\}\rangle$ satisfies the desired conditions.
\medskip

\textit{Case 2}:
Suppose $b>0$.
Then the next lemma and Lemma \ref{2-5}(ii) guarantee the existence of
a Buchsbaum complex with the desired properties.

\begin{lemma}
\label{3-10}\ 
\begin{itemize}
\item[(i)] $\Delta \cap \langle G_b\rangle = \langle \{\ob{1-b},2+p\},\{1,2+p\}\rangle$.
\item[(ii)] $(\Delta\cup \langle G_b\rangle)\cap \langle \{\ob{1-b},1,1+(2p-b)\}\rangle=\langle \{\ob{1-b},1\},\{\ob{1-b},1+(2p-b)\}\rangle$.
\end{itemize}
\end{lemma}

\begin{proof}
By using Lemmas \ref{3-5} and \ref{3-6},
one can prove $f_1(\Gamma \cup \Delta(1,b))=f_1(\Gamma)+3$
in the same way as in the proof of Lemma \ref{3-4}.
The complex $\Delta(1,b)$ has $5$ edges
$$\{1,2+p\},\{\ob{1-b},1+(2p-b)\},\{\ob{1-b},1\},\{1,1+(2p-b)\},\{2+p,1+(2p-b)\}.$$
Since the first two edges are contained in $\Sigma \subset \Gamma$,
the latter three edges are not contained in $\Gamma$.
Since the set of edges in $\Gamma$ and that of $\Delta$ are same,
these three edges are not in $\Delta$.

Then (i) holds since $\{\ob{1-b},2+p\} \in \Delta(2+p,b) \subset \Delta$,
$\{1,2+p\} \in \Sigma \subset \Delta$ and $\{\ob{1-b},1\} \not \in \Delta$,
and (ii) holds since $\{\ob{1-b},1\} \in \langle G_b\rangle$,
$\{\ob{1-b},1+(2p-b)\} \in \Sigma \subset \Delta$ and $\{1,1+(2p-b)\} \not \in \Delta \cup \langle G_b\rangle$.
\end{proof}

\section{Proof of Theorem \ref{1-2} and open problems}

To prove Theorem \ref{1-2},
we need the following easy fact:
If $\Delta$ is the disjoint union of $2$-dimensional simplicial complexes
$\Gamma$ and $\Gamma'$ then
$$h(\Delta)=h(\Gamma) +h(\Gamma')+(-1,3,-3,1).$$

\begin{proof}[Proof of Theorem \ref{1-2}]
We first prove the `only if' part.
If the vectors $(1,h_1,h_2,h_3)$ and $(1,h'_1,h'_2,h'_3)$
are $M$-vectors then $(1,h_1+h_1',h_2+h_2',h_3+h_3')$ is also an $M$-vector.
Thus, if $(1,h_1,h_2,h_3)$ and $(1,h'_1,h'_2,h'_3)$
satisfy the conditions of Theorem \ref{1-1} then $(1,h_1+h_1',h_2+h_2',h_3+h_3')$ satisfies the same conditions.
Let $\Delta$ be a $2$-dimensional Buchsbaum complex
with the connected components $\Delta_1,\dots,\Delta_{k+1}$.
Since each $\Delta_j$ is a $2$-dimensional connected Buchsbaum complex,
$$h(\Delta)=\left(1,\sum_{j=1}^{k+1} h_1(\Delta_j),\sum_{j=1}^{k+1} h_2(\Delta_j),\sum_{j=1}^{k+1} h_3(\Delta_j)\right)
+(0,3k,-3k,k)$$
satisfies the desired conditions.

Next, we prove the `if' part.
Suppose that $h'=(1,h'_1,h'_2,h'_3) \in \ZZ^4$ satisfies the conditions of Theorem \ref{1-1}.
Then there exists a $2$-dimensional Buchsbaum complex $\Delta$ with $h(\Delta)=h'$.
Let $\Gamma$ be the disjoint union of $\Delta$ and $k$ copies of $2$-dimensional simplexes.
Then $\Gamma$ is Buchsbaum and
$h(\Gamma)=h'+(0,3k,-3k,k)$ since the $h$-vector of a $2$-dimensional simplex is $(1,0,0,0)$.
\end{proof}

It will be interesting to study a generalization of Theorems \ref{1-1} and \ref{1-2} for higher dimensional Buchsbaum complexes.
On the other hand,
since properties of Buchsbaum complexes heavily depend on their Betti numbers, it might be more natural to study $h$-vectors of Buchsbaum complexes for fixed Betti numbers.
The strongest known relation between $h$-vectors and Betti numbers of Buchsbaum complexes
is the result of Novik and Swartz \cite[Theorems 3.5 and 4.3]{NS}.
In the special case when $\Delta$ is a $2$-dimensional connected Buchsbaum complex,
the result of Novik and Swartz says
\begin{eqnarray}
\label{ni}
\begin{array}{ll}
\bullet& (1,h_1(\Delta),h_2(\Delta)) \mbox{ is an $M$-vector;}\\
\bullet& h_2(\Delta) \geq 3 \beta_1(\Delta) \mbox{ and } h_3(\Delta) +\beta_1(\Delta) \leq (h_2(\Delta)-3 \beta_1(\Delta))^{\langle 2 \rangle }.
\end{array}
\end{eqnarray}
Note that $h_3(\Delta)+\beta_1(\Delta)=\beta_2(\Delta)$ in this case.
It was asked in \cite[Problem 7.10]{NS}
if there exist other restrictions on $h$-vectors of Buchsbaum complexes for fixed Betti numbers.
Recently,
Terai and Yoshida \cite[Corollary 5.1]{TY2} proved that,
for every $(d-1)$-dimensional Buchsbaum complex $\Delta$ on $[n]$,
if $d \geq 3$ and $h_0(\Delta) + \cdots +h_d(\Delta) \geq {h_1+d \choose d} -3h_1 +2$ then $\Delta$ is Cohen--Macaulay.
This result of Terai and Yoshida gives a restriction on $h$-vectors and Betti numbers of Buchsbaum complexes which does not follow from (\ref{ni}).
For example, we have

\begin{proposition}
\label{4-1}
There exist no $2$-dimensional connected Buchsbaum complexes $\Delta$ such that
$\beta_1(\Delta)=1$, $\beta_2(\Delta)=4$ and $h(\Delta)=(1,3,6,3)$.
\end{proposition}

The conditions of Betti numbers and an $h$-vector in Proposition \ref{4-1} satisfy (\ref{ni}).
However, the vector $h=(1,3,6,3)$ satisfies the assumption of the result
of Terai and Yoshida
$$h_0+h_1+h_2+h_3=13 \geq { 6 \choose 3 } -3 \times 3 +2,$$
and therefore any $2$-dimensional connected Buchsbaum complex $\Delta$
with $h(\Delta)=h$ must satisfy $\beta_1(\Delta)=0$.

It seems likely that, to characterize $h$-vectors of Buchsbaum complexes for fixed Betti numbers, we need further restrictions on $h$-vectors and Betti numbers.
Here,
we propose the following problem.

\begin{problem}
\label{pro}
Find a new inequality on $h$-vectors and Betti numbers of Buchsbaum complexes
which explain \cite[Corollary 5.1]{TY2}.
\end{problem}

\bigskip

\noindent
\textbf{Acknowledgements.}
I would like to thank the referee for pointing out that Proposition
\ref{4-1} follows from \cite[Corollary 5.1]{TY2}.
The author is supported by JSPS Research Fellowships for Young Scientists

\end{document}